\documentclass{tran-l}

\usepackage{amsmath,amssymb,amscd,amsthm,amsxtra}
\usepackage{latexsym}
\usepackage{mathrsfs}
\usepackage{graphicx, color}

\allowdisplaybreaks

\newcommand{\q}{\quad}

\newcommand{\qqq}{\quad\quad\quad}

\newcommand{\norm}[2]{{\left\| #1 \right\|}_{#2}}

\newcommand{\al}{\alpha}

\newcommand{\de}{\delta}

\newcommand{\ez}{\epsilon}

\newcommand{\la}{\lambda}

\newcommand{\cm}{\mathcal M}
\newcommand{\cb}{\mathcal B}

\newcommand{\cd}{\mathcal D}
\newcommand{\cq}{\mathcal Q}

\newcommand{\crec}{\mathcal R}

\newcommand{\lab}{\label}

\newcommand{\dint}{\displaystyle\int}

\newcommand{\f}{\frac}

\newcommand{\nf}{\infty}
\newcommand{\rrr}{\mathbf R}

\newcommand{\rn}{\mathbf R^n}

\newcommand{\ovec}{\overrightarrow}

\def\rr{{\mathbb R}}
\def\nn{{\mathbb N}}
\def\lf{\left}
\def\r{\right}
\def\hs{\hspace{0.3cm}}

\def\ls{\lesssim}

\newtheorem{theorem}{Theorem}[section]
\newtheorem{lemma}[theorem]{Lemma}
\newtheorem{proposition}[theorem]{Proposition}
\newtheorem{corollary}[theorem]{Corollary}

\theoremstyle{definition}
\newtheorem{definition}[theorem]{Definition}

\theoremstyle{remark}
\newtheorem{remark}[theorem]{Remark}

\numberwithin{equation}{section}

\begin{document}

\title{A $B_p$ condition for the strong maximal function}

\author{Liguang Liu}
\address{Liguang Liu\\
        Department of Mathematics\\
        School of Information\\
        Renmin University of China\\
        Beijing 100872\\
        P. R. China}
\email{liuliguang@ruc.edu.cn}

\author{Teresa Luque}
\address{Teresa Luque\\
Departamento De An\'alisis Matem\'atico, Facultad de Matem\'aticas,
Universidad De Sevilla, 41080 Sevilla, Spain.}
\email{tluquem@us.es}

\thanks{The first author is supported by National Natural Science
Foundation of China (Grant No. 11101425).
The second author is supported by the Spanish Ministry of
Science and Innovation Grant BES-2010-030264.}



\date{}

\dedicatory{}

\begin{abstract}
A strong version of the Orlicz maximal
operator is introduced and a natural $B_p$ condition for
the rectangle case is defined to characterize its boundedness.
 This fact let us to describe a sufficient condition for the two
 weight inequalities of the strong maximal function in terms of
 power and logarithmic bumps. Results for the multilinear version
 of this operator and for others multi(sub)linear maximal functions
 associated with bases of open sets are also studied.
\end{abstract}

\maketitle

\section{Introduction}

As it is well-known, Sawyer (\cite{S}) characterized
the pair of weights $(u,v)$ for which the Hardy-Littlewood
maximal operator, $M$, is a bounded operator from $L^p(v)$
to $L^p(u)$ for $1<p<\nf$. He showed that $M:L^p(v^p)\rightarrow L^p(u^p)$
if and only if $(u,v)$ satisfies the testing condition
\begin{equation}\label{Sp}
\sup_Q\f{\int_Q (uM(\chi_Q v^{-p'}))^pdx}{v^{-p'}(Q)}<\nf.
\end{equation}
On the other hand, it is also known that the two weight Muckenhoupt
condition $A_p$,
\[\sup_Q\lf(\f{1}{|Q|}\dint_Q u^p dx\r)^{1/p}
\lf(\f{1}{|Q|}\dint_Q v^{-p'} dx\r)^{1/p'}<\nf,\]
is necessary but not sufficient for the maximal operator
to be strong type $(p,p)$. The fact that Sawyer's condition
involves the maximal operator itself makes  it often difficult to
test in practice. Therefore, though this condition characterizes
completely the two weight problem, it would be more useful to look
for sufficient conditions close in form to the $A_p$ condition.
The first step in this direction was due to Neugebauer (\cite{N1}):
he noticed that if the pair of weights $(u,v)$ is such that for $r>1$,
\[\sup_Q\lf(\f{1}{|Q|}\dint_Q u^{pr} dx\r)^{1/pr}
\lf(\f{1}{|Q|}\dint_Q v^{-p'r} dx\r)^{1/p'r}<\nf\]
for all cubes, then
\begin{equation}\label{strong}
\dint_{\rn}\lf(uMf\r)^p dx\le C \dint_{\rn}\lf(vf\r)^p dx
\end{equation}
for all nonnegative functions $f$ and some positive constant $C$ independent of $f$.
If for a given cube $Q$ we define the normalized $L^p$ norm by
\[\norm{u}{p,Q}:=\lf(\f{1}{|Q|}\dint_Q u^p \,dx\r)^{1/p},\]
then Neugebauer's condition  can be restated in terms
of a normalized $L^p$ norm as follows:
\begin{equation}\label{Neugen}
\norm{u}{pr,Q}\norm{v^{-1}}{p'r,Q}<\nf.
\end{equation}
Notice that the $A_p$ condition can also be rewritten as
\begin{equation}\label{Ap}
\norm{u}{p,Q}\norm{v^{-1}}{p',Q}<\nf.
\end{equation}
This tell us that if we replace the average $L^p$ and
$L^{p'}$ norms in (\ref{Ap}) by some stronger ones,
then we can get a condition that is sufficient for (\ref{strong}) to hold.
At the same time, this new condition preserves the geometric structure
of the classical $A_p$. Conditions like (\ref{Neugen}) are known as
power bump conditions because the norms involved in the two weight $A_p$
condition are ``bumped up'' in the Lebesgue scale.

Motivated by \cite{N1}, \cite{CWW} and \cite{F}, P\'erez (\cite{P3}, \cite{P1})
generalized these last  conditions replacing the localized norms in (\ref{Ap})
by some other larger than the $L^p$ one, but not as big as the $L^{pr}$.
Indeed, he proved that it was enough to substitute only the norm
associated to the weight $v^{-1}$ by a stronger one defined in
terms of certain Banach function spaces $X$ with an appropriate boundedness property.

To be more precise, we let $\Phi$ be a Young function (cf. section \ref{sBp})
and  define the normalized Luxemburg norm on a cube $Q$ by
\[\norm{u}{\Phi,Q}:=
\inf\left\{\la>0:\f{1}{|Q|}\dint_Q\Phi\lf(\f{u}{\la}\r)\,dx\leq 1\right\}.\]
Associated with each Young function $\Phi$,  one can define a complementary function
\begin{equation}\label{complementary}
\bar \Phi(s):=\sup_{t>0}\{st-\Phi(t)\}
\end{equation}
for $s\ge 0$. Such $\bar \Phi$ is also a Young function and it plays,
together with the class $B_p$, an important role in the generalization
of these bump conditions. Recall that a Young function $\Phi\in B_p$
if there is a positive constant $c$ for which
\begin{equation}\label{Bp}
\dint_c^{\nf}\f{\Phi(t)}{t^p}\f{dt}{t}<\nf.
\end{equation}
This growth condition was introduced in \cite{P3} where it
was proved that if $\Phi$ is a Young function such that $\bar\Phi\in B_p$,
and $(u,v)$ is a pair of weights such that
\begin{equation}\label{Ap_M}
\sup_Q\norm{u}{p,Q}\norm{v^{-1}}{\Phi,Q}<\nf
\end{equation}
for every cube $Q$, then \eqref{strong} holds. Moreover the $B_p$
condition is sharp in the sense that: if $M$ is strong $(p, p)$
and $(u, v)$ satisfy (\ref{Ap_M}), then $\bar\Phi\in B_p$.
This result has been generalized very recently to the  more
general context of Banach function spaces by P\'erez and Mastylo (\cite{MP}).
For a more complete account of all this  we refer to the recent book \cite{CMP2}.

Besides its inherent significance for this problem,
for many other operators conditions like (\ref{Ap_M})
have resulted  in optimal sufficient
conditions for weak and strong type inequalities. In general,
given any pair of weights $(u,v)$ we will define the $A_p$ bump condition as
\begin{equation}\label{Ap_bump}
\sup_Q\norm{u}{\Psi,Q}\norm{v^{-1}}{\Phi,Q}<\nf
\end{equation}
where $\Psi$ and $\Phi$ are Young functions such that
$\bar{\Psi}\in B_{p'}$ and/or $\bar{\Phi}\in B_p$ and $Q$
is any cube $\rn$. This or related conditions are being used
in the study of operators more singular than the Hardy-Littlewood
maximal function. The first case was considered by P\'erez in \cite{P5}
for fractional integral operators where a fractional version of \eqref{Ap_bump}
was used to obtain a two weight $L^p$ estimate. The same problem for the Hilbert
transform was proved in \cite{CMP}  and by different methods in  \cite{CMP3}
for any Calder\'on-Zygmund operator with $C^1$ kernel.
Very recently the solution was extended in \cite{CRV} to
the Lipschitz case, proved full in generality in \cite{L1}
and further improved in \cite{HP} with a better control on the bounds.

The main goal of this paper is to study the two weight norm inequalities
for the boundedness of the strong maximal function
using an appropriate $B_p$ condition. We define this operator as
\begin{equation}\label{strongM}
M_\crec f(x):=\sup_{R\ni x, R\in\crec} \frac1{|R|}\int_R |f(y)|\,dy.
\end{equation}
where $f$ is a locally integrable function and the supremum
is taken over all rectangles $R$ with sides parallel to
the coordinate axes. The corresponding two weight problem
for the strong maximal function was characterized by Jawerth
(see \cite{J}) in terms of a testing condition that it is even harder
to verify than Sawyer's condition (\ref{Sp}). The problem was
also solved in \cite{P2} with a more similar approach to the
one that we study here. It was proved that if $(u,v)$ is a
couple of weights satisfying the power condition for some $r>1$
\begin{equation}\label{Ap,r}
\lf(\f{1}{|R|}\dint_R u^p\,dx\r)^{1/p}\lf(\f{1}{|R|}\dint_R v^{-p'r}\,dx\r)^{1/p'}<\nf
\end{equation}
for every rectangle $R$, and suppose
that $u^p$ satisfies the condition $(A)$: there are
constants $0 < \lambda < 1$ and $0 <   c(\lambda) < \infty$ such that
\[u^p \left( \{ x \in \rrr^{n} : M_\crec
(\chi_{E})(x)> \lambda \} \right) \le c(\la)\, u^p(E) \leqno(A)\]
for all measurable sets $E$, then
$M_\crec:L^p(v^p)\rightarrow L^p(u^p)$. In this case,
the strong weighted estimate is obtained from weak type
 ones using interpolation and the fact that there is a
 reverse H\"older's property for the weights that verifies (\ref{Ap,r}).
 However, this good property disappears if we substitute the $L^{pr}$-norm
associated to the weight $v^{-1}$ by a weaker one.
Therefore, we will need a different approach to
solve the two weight problem with general bump conditions.

To state the main result of this article, we define
the class of Young functions that enable us to obtain
bump conditions in the rectangle case.
\begin{definition}\label{d1.1}
Let $1<p<\infty$. A Young function $\Phi$ is said
to satisfy the strong $B_p^{*}$ condition,
if there is a positive constant $c$ such that
\begin{equation}\label{Strong-Bp}
\int_c^\nf \frac{\Phi_n(\Phi(t))}{t^{p}}\,\frac{dt}t<\nf,
\end{equation}
where $\Phi_n(t):= t[\log(e+t)]^{n-1}\sim
t[1+(\log^+t)^{n-1}]$ for all $t>0$.
In this case, we say that $\Phi\in B_p^\ast$.
\end{definition}

Then we have the following result.
\begin{theorem}\label{Two-weight}
Let $1<p<\infty$, and let $\Phi$ be a Young function
such that the complementary Young function $\overline{\Phi}$
satisfies the condition (\ref{Strong-Bp}).
\begin{itemize}
\item[(i)] Let $(u,v)$ be a couple of weights such that
$u^p$ satisfies the condition $(A)$ and
\begin{equation}\label{2-weight}
\left(\frac{1}{|R|}\displaystyle\int_R u^p\right)^{1/p} \norm{v^{-1}}{\Phi, R}\leq K,
\end{equation}
for some positive constant $K$ and for all rectangles $R$. Then
there is a constant $C$ such that
\[\int_{\mathbb{R}^n}\left(uM_{\crec}f\right)^p\,dx\leq
C\int_{\mathbb{R}^n}\left(vf \right)^p\,dx,\]
for all non-negative functions $f$.
\item[(ii)] Condition (\ref{Strong-Bp}) is also necessary.
That is, suppose that $\Phi$ has the property that
$M_{\crec}:\,L^p(v^p)\rightarrow L^p(u^p)$ whenever
the couple of weights $(u,v)$ satisfies
\begin{equation*}
\left(\frac{1}{|R|}\displaystyle\int_R u^p\right)^{1/p} \norm{v^{-1}}{\Phi, R}\leq K,
\end{equation*}
for some positive constant $K$ and for all rectangles $R$; then $\bar\Phi\in B_p^*$.
\end{itemize}
\end{theorem}

If this result is compared with the analogous one for
the Hardy-Littlewood operator (\cite[Theorem 1.5]{P3}),
then there is a key difference between them. For the former
not only do we need a more restrictive class of young functions
(the class $B_p^*$), but also it is necessary to ask for
an extra condition $(A)$ on the weight $u$. To understand the role of
this extra condition $(A)$, we should keep in mind the next relevant fact.
The study of the boundedness properties of a maximal operator with respect
to a family of bounded measurable sets is closely connected to
studying the covering properties of that family (cf. \cite{G}).
But since the geometry of rectangles in $\rn$ is much more
intricate than that of cubes in $\rn$, the classical covering
lemmas don't work in the rectangle case. Particularly, the
Calder\'on-Zygmund decomposition that  is strongly used
in the cube case cannot be used here. In this sense, the
condition $(A)$ is necessary to deal with rectangles and
with their covering properties (see Lemma \ref{scatteredproperty} below).
The problem with this condition is that, as happened with Sawyer's condition,
it involves itself the operator and therefore it would be useful
to have a replacement that did not. The $A_\nf$ condition would
be a good candidate since it is simpler than the $(A)$ condition;
however, it is also stronger (see for example \cite[p.\,1123]{P2}).
Unfortunately, repeated efforts to get a weaker condition and a
simpler covering argument have failed, but we believe that
such a result would be very interesting and would provide new
insights about the study of the strong maximal operator.

In this article, we also address similar questions
involving the multilinear version of the strong maximal
function and some other more general maximal functions.
We define the strong multilinear maximal function as
\begin{equation}\label{multilinearRectangles}
\cm_{\crec}(\vec f\,)(x):=\sup_{R\ni x}
\prod_{j=1}^m\frac1{|R|}\int_R |f_j(y)|\,dy, \hspace{0.5cm}x\in\rn
\end{equation}
where $\vec f=(f_1,\cdots,f_m)$ is an $m$-dimensional vector
of locally integrable functions and where the supremum
is taken over all rectangles with sides parallel to the
coordinate axes. This multilinear maximal operator was defined
for first time by Lerner et al. in \cite{LOPTT}  with the usual
cubes instead of rectangles. In that paper it is shown that this
operator plays a central role in the theory of multilinear
Calder\'on-Zygmund operators.
The operator \eqref{multilinearRectangles}  as well as its version
for a general basis $\cb$ (cf. section \ref{s_basis}) was introduced
and studied in \cite{GLPT}.
In particular, it was shown the weak boundedness of $\cm_{\crec}$
whenever the weights satisfy a certain power bump
variant of the multilinear version of the $A_p$ condition.
That is, for $1<p_1,\cdots, p_m<\infty$ and $0<p<\infty$
such that $\f1p=\sum_{j=1}^m \f1{p_j}$, the multilinear strong maximal function maps
\[L^{p_1}(w_1)\times\cdots\times L^{p_m}(w_m)\rightarrow L^{p,\nf}(\nu)\]
provided that $(\nu, \vec w)=(\nu, w_1,\dots, w_m)$ are
weights that satisfy the power bump condition
\begin{equation}\label{bumpedap_R}
\sup_{R\in \crec} \lf(\f 1 {|R|} \int _R \nu(x)\, dx \r)\, \prod_{j=1}^m
\left( \f 1 {|R|} \int _Rw_j^{(1-p'_j)r}\,
dx\right)^{\frac{p}{p_j'r}}<\infty\,
\end{equation}
for some $r>1$.
In the case that  $\nu= \prod_{j=1}^m w_j^{p/p_j}$
the strong boundedness of $\cm_{\crec}$ is also characterized;
see \cite[Corollary~2.4 and Theorem~2.5]{GLPT}.
Multiple weight theory that adapted to the basis $\cb=\cq$
and to its multilinear operator associated, $\cm_\cq$,
has been fully  developed by Lerner et al. \cite{LOPTT}
and generalized very recently by Moen \cite{M}.

Inspired by these previous works, we will introduce
the multilinear version of  \eqref{Ap_M} and \eqref{2-weight} for
weights $(\nu, \vec w)$ associated with general basis.
Then, the $L^{p_1}(w_1)\times\cdots\times L^{p_m}(w_m)\rightarrow L^{p}(\nu)$
boundedness of $\cm_\cb$ (cf. section \ref{s_basis}) will be proved
whenever $\nu$ is any arbitrary weight
such that $\nu^p$ satisfies condition $(A)$. This result is given in
Theorem \ref{t1}. As an application of this theorem, we will
obtain the strong version of \cite[Theorem~2.3]{GLPT}
and we will deduce the analogous result of \cite[Theorem 6.6]{M}
for the strong multilinear maximal function. See Corollary \ref{c1} and Corollary \ref{c2}.

The general organization of this paper is  as follows. Section \ref{sBp} contains
some preliminaries about Orlicz spaces and a characterization of
the strong $B_p^*$ condition. Section \ref{s_basis}
presents some definitions about general basis and the statement
of the main strong weight result for a general multilinear
operator (Theorem \ref{t1}). Also, we give the proofs of Corollaries \ref{c1}
and  \ref{c2}, and we will deduce the proof of Theorem \ref{Two-weight}
by applying Corollary \ref{c2}. Finally, the last section shows the proof of Theorem \ref{t1}.

\medskip

\noindent{\bf Acknowledgments.}
The authors are very grateful to Professor Carlos P\'erez for
suggesting the problem and for some valuable discussions on the
subject of this paper. We also wish to thank
the referee for his/her valuable corrections and comments on the paper.


\section{Characterization of the $B_p^\ast$ condition}\label{sBp}

\quad To present this characterization we need to recall a few
facts about  Orlicz spaces and we shall refer the reader to
\cite[Chapter 5]{CMP2} and \cite{RR} for complete details.
A function $\Phi:[0,\infty)\rightarrow [0,\infty)$ is a Young
function if it is continuous, convex and strictly increasing
satisfying $\Phi(0)=0$ and $\Phi(t)\rightarrow \infty$ as $t\rightarrow\infty$.
A Young function $\Phi$ is said to be doubling if there exists a positive
constant $\al$ such that
\[\Phi(2t)\leq\al\Phi(t)\]
for all $t\ge 0$. The normalized $\Phi$-norm of a function $f$ over
a set $E$ with finite measure
is defined by
\[\|f\|_{\Phi, E}:=
\inf \lf\{\la>0\,:\, \frac{1}{|E|}\int_E  \Phi \left
(\frac{|f(x)|}{\la }\right )dx\leq 1\r\}.\]
The complementary of the Young function (\ref{complementary}) has properties
\begin{eqnarray}\label{p_complementary}
\Phi^{-1}(t)\bar \Phi^{-1}(t)\sim t \quad\textup{for all}\, \, t\in(0, \nf)
\end{eqnarray}
and
\[st \le C\, \Big[ \Phi(t) +\bar \Phi(s) \Big]\]
for all $s,t\ge 0$. Also the $\bar \Phi$-norms  are related to the
$L_{\Phi}$-norms  via the  {\it the generalized H\"older's
inequality}, namely
\begin{equation}\label{Holder-Orlicz}
\frac1{|E|}\,\int_{E}|f(x)\,g(x)|\,dx \le
2\,\|f\|_{\Phi,E}\,\|g\|_{\bar \Phi, E}.
\end{equation}

Consider the Orlicz maximal operator
$$M_\Phi^\cq  f(x):=\sup_{Q\ni x, Q\in\cq}\|f\|_{\Phi, Q},$$
where the supremum is taken over all cubes containing $x$.
P\'erez \cite[Theorem~1.7]{P3} proved the following key observation: when $1<p<\nf$
and  $\Phi$ is a doubling Young function, then
$$
M_\Phi^\cq: L^p(\rn) \longrightarrow  L^p(\rn)
\quad \mbox{if and only if}\quad \Phi \,\mbox{ satisfies \eqref{Bp}}.
$$

Here we remark that the hypothesis of $\Phi$ being doubling was
only used to prove the necessity of the $B_p$ condition  but
we show now that can be removed. Indeed, if we assume that for
any non-negative function $f$ the operator $M_\Phi^\cq$ is
bounded on $L^p(\rn)$ and we take $f=\chi_{[0,1]^n}$, we have
\[\int_{\rn}M_\Phi^\cq(\chi_{[0,1]^n})(y)^pdy<\infty.\]
Now, it is easy to see that there exists a positive dimensional
constant $b$ such that whenever $|y|>1$
\[M_\Phi^\cq(\chi_{[0,1]^n})(y)=\frac{1}{\Phi^{-1}(\frac{|y|^n}{b})}.\]
Hence
\begin{eqnarray*}
\dint_{\rn}M_\Phi^\cq(\chi_{[0,1]^n})(y)^pdy
&\geq&p\int_{0}^{\infty}t^p\left|\left\{y\in\mathbb{R}^n:|y|>1,
\frac{1}{\Phi^{-1}(\frac{|y|^n}{b})}>t\right\}\right|\f{dt}{t}\\
&=&
p\int_{0}^{\infty}t^p\left|\left\{y\in\mathbb{R}^n:1<|y|<\Phi
\left(\frac{1}{t}\right)^{1/n}b^{1/n}\right\}\right|\f{dt}{t}\\
&=&
c_n p\int_{0}^{\infty}t^p
\left(b \Phi\left(\frac{1}{t}\right)-1\right)\f{dt}{t},
 \end{eqnarray*}
 where $c_n$ is a positive constant depending only on $n$.
Since $\Phi$ is increasing and $\Phi(t)\rightarrow \infty$
as $t\rightarrow\infty$, we can choose some $t_0>0$ such that for every $t\leq t_0$,
\[b \Phi\left(\frac{1}{t}\right)-1\geq \frac{b}{2}\Phi\left(\frac{1}{t}\right).\]
Then
\begin{eqnarray*}
\infty &&>c_n p\int_{0}^{\infty}t^{p-1}\left(b \Phi\left(\frac{1}{t}\right)-1\right)dt\\
&&\geq
\frac{c_npb}{2}\int_{0}^{t_0}t^{p-1}\Phi\left(\frac{1}{t}\right)dt
=\frac{c_n pb}{2}\int_{1/t_0}^{\infty}\frac{\Phi(t)}{t^p}\frac{dt}{t}.
\end{eqnarray*}

Motivated by P\'erez \cite[Theorem~1.7]{P3} and the previous observation, in this section we
consider the Orlicz maximal operator $M_\Phi^\crec $ associated with rectangles
rather than cubes.
Precisely, for each locally integrable function $f$
and Young function $\Phi$ we define the {\em Orlicz maximal operator} $M_\Phi^\crec $ by
$$M_\Phi^\crec  f(x):=  \sup_{R\ni x, R\in\crec}\|f\|_{ B, R}$$
where the supremum is taken over all rectangles with
sides parallel to the coordinate axes containing $x$.
In particular, when $\Phi(t)=t$ the maximal operator
$M_\Phi^\crec $ is exactly the classical
{\em strong maximal function} (\ref{strongM}).

The next characterization shows that the boundedness
of $M_\Phi^\crec$ is closely connected with the class $B_p^*$.

\begin{theorem}\label{t2}
Let $1<p<\nf$. Suppose that $\Phi$ is a Young function.
Then the following statements are equivalent:
\begin{enumerate}
\item[(i)] $\Phi\in B_p^{*}$;
\item[(ii)] the operator $M_\Phi^\crec $ is bounded on $L^p(\rn)$;
\item[(iii)] there exists a positive constant $C$ such that
\begin{equation}\label{M_phi-bdd-characterization}
\int_{\rn} [M_\crec (f)(y)]^p\frac{1}{[M_{\bar \Phi }^\crec(u^{1/p})(y)]^p}\,dy
\le C\int_{\rn} f(y)^p \frac1{u(y)}\,dy
\end{equation}
for all non-negative functions $f$ and $u$;
\item[(iv)]  there exists a positive constant $C$ such that
for all non-negative functions $f$ and
all $w$ satisfying the condition $(A)$,
\begin{equation}\label{w-char}
\int_{\rn} [M_\Phi^\crec (f)(y)]^pw(y)\,dy\le C\int_{\rn} f(y)^p M_\crec w(y)\,dy.
\end{equation}
\end{enumerate}
\end{theorem}
As particular examples of Young functions $\Phi\in B_p^*$,
one can easily check that a Young function $\Phi$ satisfies
 the condition (\ref{Strong-Bp}) if
$$\Phi(t)\sim t^\alpha \log^{-\beta} (e+t)
\qquad  -\infty<\alpha<p,\, \beta\in\mathbf R;$$
$$ \Phi(t)\sim t^{p} \log^{-\beta} (e+t)
\qquad\qquad \beta>n;$$
or the weaker one
$$\Phi(t)\sim t^{p} \log^{-n} (e+t)\, [\log\log(e+t)]^{-\gamma}
\qquad \gamma>1.$$

\begin{proof}[Proof of Theorem \ref{t2}]
We assume that (i) holds and show (ii). To this end, for each $t>0$,
we split the function $f$ into $f=f_t+f^t$
with
$f_t:= f\chi_{|f|>t/2}$ and $f^t:= f\chi_{|f|\le t/2}.$
Then, $$M_\Phi^\crec  f\le M_\Phi^\crec (f_t)+M_\Phi^\crec (f^t)\le M_\Phi^\crec (f_t)+t/2$$
and
$$\{x\in\rn:\, M_\Phi^\crec  f(x)>t\}\subset \{x\in\rn:\, M_\Phi^\crec  (f_t)(x)>t/2\}.$$
Set
$$\Omega_t:= \{x\in\rn:\, M_\Phi^\crec  (f_t)(x)>t/2\}.$$
 Choose a compact set $K\subset\Omega_t$ such that
$|\Omega_t|/2\le|K|\le|\Omega_t|$. There exists a sequence of rectangles $\{R_j\}_{j=1}^N$
such that $K \subset\cup_{j=1}^N R_j$ and
$$\|f_t\|_{\Phi, R_j}>t\qquad\forall\, j\in\{1,\cdots,N\}.$$
By \cite[Lemma~6.1]{GLPT}, the condition $\|f_t\|_{\Phi, R_j}>t$ implies that
$$1<\norm{\f{f_t}{t}}{\Phi, R_j}\le \f1{|R_j|}\int_{R_j} \Phi\lf(\f{|f_t(y)|}{t}\r)\,dy.$$
Applying now the covering lemma from \cite{cf} (see also
\cite[Theorem~4.1 (C)]{Bagby}), there are dimensional
positive constants $\de,c$ and a subfamily
$\{\widetilde R_j\}_{j=1}^{\ell}$ of $\{R_{j}\}_{j=1}^N$ satisfying
$$\bigg|\bigcup_{j=1}^N R_{j}\bigg|\leq c\,
\bigg|\bigcup_{j=1}^\ell \widetilde R_j\bigg|,$$
and
$$
\int_{\bigcup_{j=1}^\ell \widetilde R_j} \exp
\bigg(\de\,\sum_{j=1}^\ell \chi_{\widetilde
R_j}(x)\bigg)^{\f1{n-1}}\,dx \le 2\bigg|\bigcup_{j=1}^\ell
\widetilde R_j\bigg|.
$$
Let $\widetilde E:= \bigcup_{j=1}^\ell \widetilde R_j$.
Recall that for each $\theta>0$, there exists a constant
 $C_\theta$ such that for all
$a, b\ge0$,
$$ab\le C_\theta(e^{\theta a}-1+b[1+(\log_+b)^{n-1}])
=C_\theta(e^{\theta a}-1+ \Phi_n(b));$$
see \cite[p.\,887]{Bagby}. Then, for all $\ez>0$,
\begin{eqnarray*}
|\widetilde E|
&&\leq \,\sum_{j=1}^\ell |\widetilde R_j|\\
&&\le  \sum_{j=1}^\ell \int_{\widetilde R_j} \Phi\lf(\f{|f_t(y)|}{t}\r)\,dy\\
&&=\int_{\bigcup_{j=1}^\ell \widetilde R_j}
\sum_{j=1}^\ell \chi_{\widetilde R_j}(y)\Phi\lf(\f{|f_t(y)|}{t}\r)\,dy\\
&&\le \ez C_\de \int_{\bigcup_{j=1}^\ell \widetilde R_j}
\lf[\exp\lf(\de\sum_{j=1}^\ell \chi_{\widetilde R_j}(y)\r)-1
+  \Phi_n\lf(\f1\ez\Phi\lf(\f{|f_t(y)|}{t}\r)\r)\r]\,dy\\
&&\le \ez C_\de\lf\{|\widetilde E|+ \Phi_n(1/\ez)
\int_{\widetilde E} \Phi_n\lf(\Phi\lf(\f{|f_t(y)|}{t}\r)\r)\,dy\r\}.
\end{eqnarray*}
Choosing $\ez>0$ small enough we obtain
$$|\widetilde E|\le  C \int_{\widetilde E}
\Phi_n\lf(\Phi\lf(\f{|f_t(y)|}{t}\r)\r)\,dy.$$
Since $|\Omega_t|\sim|K|$ and $|K|\le c |\widetilde E|$
and $ \Phi_n(\Phi(0))=0$, it follows that
\begin{eqnarray*}
|\Omega_t|
&&\le C \int_{\widetilde E}
\Phi_n\lf(\Phi\lf(\f{|f_t(y)|}{t}\r)\r)\,dy
\le C \int_{\{y\in\rn: |f(y)|>t/2\}}
\Phi_n\lf(\Phi\lf(\f{|f(y)|}{t}\r)\r)\,dy.
\end{eqnarray*}
This inequality  and the fact
$\{x\in\rn:\, M_\Phi^\crec  f(x)>t\}\subset\Omega_t$,
together with the change of variable $s=|f(y)|/t$, yields
\begin{eqnarray*}
\|M_\Phi^\crec  f\|_{L^p(\rn)}^p
&&=p\dint_0^\nf t^{p}|\{x\in\rn:\, M_\Phi^\crec  f(x)>t\}|\,\f{dt}{t}\\
&&\le p\int_0^\nf t^{p}|\Omega_t|\f{dt}{t}\\
&& \le C \int_0^\nf \int_{\{y\in\rn:|f(y)|>t/2\}} t^{p}
\Phi_n\lf(\Phi\lf(\f{|f(y)|}{t}\r)\r)\,dy\,\f{dt}{t}\\
&&= C \int_{\rn}\int_0^{2|f(y)|} t^{p}
\Phi_n\lf(\Phi\lf(\f{|f(y)|}{t}\r)\r)\,\f{dt}{t}\,dy\\
&&\le C \int_{\rn}\int_{1/2}^\nf |f(y)|^p
\f{ \Phi_n(\Phi(s))}{s^{p}}\,\f{ds}{s}\,dy\\
&&\le C\|f\|_{L^p(\rn)}^p,
\end{eqnarray*}
where in the last step we use the hypothesis $\Phi\in B_p^\ast$.
This proves (ii).

Let us assume that (ii) holds. Using the
generalized H\"older's inequality \eqref{Holder-Orlicz}
we obtain
$$M_\crec (hg)(y)\le M_\Phi^\crec  h(y) M_{\bar \Phi }^\crec g(y),$$
which together with the boundedness of $M_\Phi^\crec$ on $L^p(\rn)$
implies that
\begin{eqnarray*}
\int_{\rn} [M_\crec (hg)(y)]^p\frac{1}{[M_{\bar \Phi }^\crec g(y)]^p}\,dy
&&\le \int_{\rn} [M_\Phi^\crec  h(y)]^p\,dy \\
&&\le \|M_\Phi^\crec \|_{L^p(\rn)\to L^p(\rn)}^p \int_{\rn} h(y)^p\,dy,
\end{eqnarray*}
from which we obtain the claim (iii) by taking $h= fu^{-1/p}$ and $g=u^{1/p}$.

To prove that (iii) implies (i), for any $N\in\nn$,
 we let $f:= \chi_{[0, 1]^n}$
and $u_N:= \chi_{[0, 1]^n}+\frac{\chi_{\rn\setminus [0,1]^n}}{N}$
in \eqref{M_phi-bdd-characterization}. Hence we get
\begin{eqnarray*}
\int_{\rn} \lf(\frac{M_\crec  (\chi_{[0, 1]^n})(y)}
{M_{\bar \Phi }^\crec (\chi_{[0, 1]^n}
+\frac{\chi_{\rn\setminus [0,1]^n}}{N})(y)}\r)^p\,dy\le C.
\end{eqnarray*}
Observing that $ M_{\bar \Phi }^\crec (f+g)
\le M_{\bar \Phi }^\crec f+M_{\bar \Phi }^\crec g$
and using the monotone convergence lemma,
we deduce
\begin{eqnarray}\label{e1}
\int_{\rn} \lf(\frac{M_\crec  (\chi_{[0, 1]^n})(y)}{M_{\bar \Phi }^\crec
(\chi_{[0, 1]^n})(y)}\r)^p\,dy\le C.
\end{eqnarray}
It is easy to see that for any point $(y_1, \cdots, y_n)\in\rn$
such that $y_j>1$ for all $j\in\{1,\cdots,n\}$,
we have
$$M_\crec  (\chi_{[0, 1]^n})(y)=\sup_{R\ni y, R\in\crec} \frac{|R\cap[0, 1]^n|}{|R|}
=\frac 1 {y_1y_2\cdots y_n},$$
and
\begin{eqnarray*}
M_{\bar \Phi }^\crec (\chi_{[0, 1]^n})(y)
&&=\sup_{R\ni y, R\in\crec} \inf\lf\{\la>0:\,
 \bar \Phi \lf(\la^{-1}\r)\le \frac{|R|}{|R\cap[0, 1]^n|}\r\}\\
&&=\sup_{R\ni y, R\in\crec} \frac1{ \bar \Phi ^{-1}
\lf(\frac{|R|}{|R\cap[0, 1]^n|}\r)}\\
&&=\frac1{\bar \Phi ^{-1}\lf(y_1y_2\cdots y_n\r)}.
\end{eqnarray*}
Inserting these two estimates into \eqref{e1}
and using \eqref{p_complementary},  we deduce that
\begin{eqnarray*}
\nf&&> \int_1^\nf\cdots\int_1^\nf
\lf(\frac{\bar \Phi ^{-1}\lf(y_1y_2\cdots y_n\r)}
{y_1y_2\cdots y_n}\r)^p\,dy_n\,\cdots dy_1\\
&&\sim \int_1^\nf\cdots\int_1^\nf
\lf(\frac1{\Phi^{-1}\lf(y_1y_2\cdots y_n\r)}\r)^p\,dy_n\,\cdots dy_1.
\end{eqnarray*}
Then it follows that
\begin{eqnarray*}
\int_1^\nf
\lf(\frac1{\Phi^{-1}\lf(y_1y_2\cdots y_n\r)}\r)^p\,dy_n
&&= \frac{1}{y_1\cdots y_{n-1}} \int_{\Phi^{-1}(y_1\cdots y_{n-1})}^\nf
\frac {\Phi'(z)}{z^p}\,dz\\
&&\geq\frac{1}{y_1\cdots y_{n-1}} \int_{\Phi^{-1}(y_1\cdots y_{n-1})}^\nf
\frac {\Phi(z)}{z^{p+1}}\,dz,
\end{eqnarray*}
where we have used the fact that $\Phi'(t)\geq\f{\Phi(t)}{t}$
for any Young function $\Phi$. Now we take the integral in the
variable $y_{n-1}$ and we obtain
\begin{eqnarray*}
&&\int_1^\nf \int_1^\nf
\lf(\frac1{\Phi^{-1}\lf(y_1y_2\cdots y_n\r)}\r)^p\,dy_n\,dy_{n-1}\\
&&\hs\geq\int_1^\nf \frac{1}{y_1\cdots y_{n-1}}
\int_{\Phi^{-1}(y_1\cdots y_{n-1})}^\nf \frac {\Phi(z)}{z^{p+1}}\,dz \,dy_{n-1}\\
&&\hs= \int_{\Phi^{-1}(y_1\cdots y_{n-2})}^\nf \int_1^{\f{\Phi(z)}{y_1\cdots y_{n-2}}}
\frac{1}{y_1\cdots y_{n-1}} \,dy_{n-1} \frac {\Phi(z)}{z^{p+1}}\,dz\\
&&\hs= \frac{1}{y_1\cdots y_{n-2}} \int_{\Phi^{-1}(y_1\cdots y_{n-2})}^\nf
\ln\lf(\f{\Phi(z)}{y_1\cdots y_{n-2}}\r)
\frac {\Phi(z)}{z^{p+1}}\,dz.
\end{eqnarray*}
Moreover,
\begin{eqnarray*}
&&\int_1^\nf \frac{1}{y_1\cdots y_{n-2}}
\int_{\Phi^{-1}(y_1\cdots y_{n-2})}^\nf \ln\lf(\f{\Phi(z)}{y_1\cdots y_{n-2}}\r)
\frac {\Phi(z)}{z^{p+1}}\,dz\,dy_{n-2}\\
&&\hs= \frac{1}{y_1\cdots y_{n-3}}
\int_{\Phi^{-1}(y_1\cdots y_{n-3)}}^\nf
\int_1^{\f{\Phi(z)}{y_1\cdots y_{n-3}}} \f1{y_{n-2}}
\ln\lf(\f{\Phi(z)}{y_1\cdots y_{n-2}}\r)\,dy_{n-2}
\frac {\Phi(z)}{z^{p+1}}\,dz\\
&&\hs= \frac{1}{y_1\cdots y_{n-3}}
\int_{\Phi^{-1}(y_1\cdots y_{n-3)}}^\nf \lf(\ln\lf(\f{\Phi(z)}{y_1\cdots y_{n-3}}\r)\r)^2
\frac {\Phi(z)}{z^{p+1}}\,dz.
\end{eqnarray*}
We iterate this process by integrating over the next
variables $y_{n-3},\cdots, y_1$ in turn
and we obtain
\begin{eqnarray*}
\nf&&> \int_1^\nf\cdots\int_1^\nf
\lf(\frac1{\Phi^{-1}\lf(y_1y_2\cdots y_n\r)}\r)^p\,dy_n\,\cdots dy_1\\
&&\geq\int_{\Phi^{-1}(1)}^\nf \lf(\ln\lf(\Phi(z)\r)\r)^{n-1}\frac {\Phi(z)}{z^{p+1}}\,dz\\
&&\geq\int_{\Phi^{-1}(e)}^\nf \f{ \Phi_n(\Phi(z))}{z^{p+1}}\,dz,
\end{eqnarray*}
which proves (i).

To conclude the proof of this theorem, note that (iv) implies (ii)
by choosing $w=1$ in the right side of \eqref{w-char}.
In order to prove that (ii) implies (iv) we will
proceed using an argument very similar to
the one presented in the proof of Theorem \ref{t1}.
For this reason, we will give the details to complete this
proof in the fourth section, Remark \ref{r_t2}.
 Here we  point out that the proof of Theorem \ref{t1} below is independent
of Theorem \ref{t2}.
\end{proof}

It should be remarked that the classical  $B_p$ condition \eqref{Bp}
is not sufficient
for the $L^p(\rn)$-boundedness of $M_{\Phi}^\crec$. For this,
we consider for example the function
$f=\chi_{[0,1]^n},$
and
\[\Phi(t) = \frac{t^p}{(\mathrm{log} (1 + t))^{1 + \delta}}
\qquad \textup{for all}\,\, t\in(0,\infty)\]
with $0<\delta<1$.
It is easy to verify that such a function $\Phi$ satisfies \eqref{Bp} but
fails for \eqref{Strong-Bp}.
For simplicity, we consider only the case when $n=2$.
If $|x_i|> >1$ (we can take $|x_i|>4$) for $i=1, 2$, then
\[M_\Phi^\crec f(x_1,x_2)=\frac{1}{\Phi^{-1}(|x_1||x_2|)}
\sim \frac1{{|x_1||x_2|}^{1/p}(\mathrm{log}
(1 + |x_1||x_2|))^{(1 + \delta)/p}},\]
by terms of the fact that $\Phi^{-1}(t)
\sim t^{1/p}(\mathrm{log} (1 + t))^{(1 + \delta)/p}$
for all $t\in(0,\infty)$.
Using Fubini's theorem, we have
\begin{eqnarray*}
\int_{\mathbf{R}^2}\left( M_\Phi^\crec f(x) \right)^p \,dx
& \geq &   \int_{4}^{\infty}\int_{4}^{\infty}
\left(\frac{1}{\Phi^{-1}(|x_1||x_2|)}\right)^p dx_2dx_1\\
& \geq&
  \int_{4}^{\infty}\int_{4}^{\infty}\frac{1}{x_1x_2\left(\mathrm{log}
  \left(1 +x_1x_2\right)\right)^{1 + \delta}} dx_2dx_1\\
& \geq &
  \int_{4}^{\infty}\int_{4}^{\infty}\frac{1}{\left(1+x_1x_2\right)
  \left(\mathrm{log} \left(1 + x_1x_2\right)\right)^{1 + \delta}} dx_2dx_1\\
& \sim &
  \frac{1}{\delta}\int_{4}^{\infty}\frac{1}{x_1(\mathrm{log} (1 + 4x_1))^{\delta}}dx_1\\
& \geq &
  \frac{1}{\delta}\int_{16}^{\infty}\frac{1}{(1+x_1)
  (\mathrm{log} (1 + x_1))^{\delta}}dx_1=\infty.
 \end{eqnarray*}
However,
$$\|f\|_{L^p(\mathbf R^2)}=\|\chi_{[0,1]^2}\|_{L^p(\mathbf R^2)}=1.$$
Hence, $M_\Phi^\crec$ is not bounded on $L^p(\mathbf R^2)$.
The general case for $n>2$ is similar and we omit the details.

Though the  $B_p$ condition \eqref{Bp} is not sufficient
for the $L^p(\rn)$-boundedness of $M_{\Phi}^\crec$, we can
remedy this situation if we restrict to those Young functions
that are submultiplicative; see Proposition \ref{p4.2} below.
We say that a Young function $\Phi$ is submultiplicative if for each $t,s> 0$,
\[\Phi(ts)\leq\Phi(t)\Phi(s).\]

\begin{proposition}\label{p4.2}
Let $1<p<\nf$. Assume that $\Phi$ is a submultiplicative
Young function such that $\Phi\in B_p$. Then the operator
$M_\Phi^\crec $ is bounded on $L^p(\rn)$.
\end{proposition}
\begin{proof}
This is a simple consequence of the fact that for a  submultiplicative
Young function such that $\Phi\in B_p$, there exits $\ez>0$
for which $\Phi\in B_{p-\ez}$(\cite[Lemma 4.3]{P3}). Indeed,
using the previous theorem we only need to prove that
$\Phi\in\Phi_p^*$. Note that
\begin{equation}\label{2.6}
\dint_{c}^\nf \f{ \Phi_n(\Phi(s))}{s^{p}}\f{ds}{s}
=\dint_{c}^\nf \f{ \Phi(s)}{s^{p}}\f{ds}{s}
+\dint_{c}^\nf \f{ \Phi(s)}{s^{p-\ez}}
\f{(\log^+\Phi(s))^{n-1}}{s^\ez}\f{ds}{s}.
\end{equation}
It is clear that the first term in the right hand  of
\eqref{2.6} is bounded. On the other hand,
\[\f{(\log^+\Phi(s))^{n-1}}{s^\ez}
\leq\f{\Phi(s)^{\de(n-1)}}{s^\ez\de^{(n-1)}},\]
with $\de>0$. Since $\Phi$ is  in the class $B_p$,
it follows that $\Phi(t)\ls t^p$ for $t\geq 1$ and hence
 for $\de=\f{\ez}{p(n-1)}$ the above term is bounded.
This further implies that the second term of \eqref{2.6} is bounded by
a constant multiple of
$$\int_{c}^\nf \f{ \Phi(s)}{s^{p-\ez}}\f{ds}{s},$$
which together with the aforementioned fact that  $\Phi\in B_{p-\ez}$
gives the boundedness of the second term of \eqref{2.6}.
\end{proof}

\begin{remark}\rm
We observe that a typical Young function that belongs
to the class $B_p$ and that it is also submultiplicative
is  $\Phi(t)=t^r$ with $1\leq r<p$. Another more
interesting example is the function $\Phi$ given by
$\Phi(t)=t^r(1+\log_+t)^{\al}$  with $1\leq r<p$ and $\al>0$.
It is not difficult to see that such functions are
submultiplicative and they are
 in the $B_p$ class. See Cruz-Uribe and Fiorenza \cite{CF}
 for related discussions on the topic of
submultiplicative  Young functions.
\end{remark}


\section{Weighted theory for general
 bases and proof of Theorem \ref{Two-weight}}\label{s_basis}
We start by introducing some notation that we will use
through this section. By a {\it basis} $\cb$ in
$\rn$ we mean a collection of open sets in $\rn$.
The most important examples of bases arise by taking
$\cb= \cq$ the family of all open cubes in $\rn$ with sides parallel
to the axes, $\cb= \cd$ the family of all open dyadic cubes in
$\rn$,  and   $\cb = \crec$ the family of all open rectangles in
$\rn$ with sides parallel to the axes.   Another interesting example
is the set
 $\Re$ of all  rectangles in $\rr^3$ with
sides parallel to the coordinate axes whose side lengths
are $s$, $t$, and $st$, for some $t,s>0$.

Assume that $\cb$ is a basis and that $\{\Psi_j\}_{j=1}^m$
is a sequence of Young functions, we define the multi(sub)linear
Orlicz maximal function by
$$\cm_{ \overrightarrow \Psi}^\cb(\vec f\,)(x)
:=\sup_{B\in\cb, B\ni x}\prod_{j=1}^m\|f_j\|_{\Psi_j, B}.$$
In particular, when $\Psi_j(t)= t$ for all $t\in(0, \infty)$ and all $j\in\{1,\cdots,m\}$, we simply write
$\cm_{ \overrightarrow \Psi}^\cb$ as $\cm_\cb$; that is,
$$\cm_{\cb}(\vec f\,)(x)=\sup_{B\in\cb, B\ni x}\prod_{j=1}^m\frac1{|B|}\int_B |f_j(y)|\,dy.$$
When $m=1$, we use $M_\Psi^{\cb}$ and $M_{\cb}$ to respectively denote
$\cm_{ \overrightarrow \Psi}^\cb$ and $\cm_{\cb}$.

We  say that $w$
is a {\it weight} associated with the basis $\mathcal B$ if $w$ is a
non-negative measurable function in $\rn$ such that $w (B) =
\int_{B} w(y)\, dy < \infty$ for each $B$ in $\mathcal B$.
A weight $w$  associated with $\cb$ is said to satisfy the
$A_{p,\cb}$ condition, $1<p<\infty$, if
\[\sup_{B\in \cb} \left(\f 1 {|B|} \int _B  w \, dx \right)\,  \left( \f 1 {|B|}
\int _Bw^{1-p'}\, dx\right)^{\frac{p}{p'}} <\infty\ .\]
In the limiting case $p = 1$ we say that $w$ satisfies the $A_{1, \mathcal B}$
if
$$
\left( \frac{1}{ |B| } \int_{B} w(y)\, dy \right) \mathop\mathrm{esssup}_{B}\, w^{-1} \le c
$$
for all $B \in \mathcal B$; this is equivalent to
$
 M_{{\mathcal B}}w(x) \le c\, w (x)
$
for almost all $x \in \rn$. It follows from these definitions and
H\"older's inequality that
$
A_{p, \mathcal B} \subset   A_{q, \mathcal B}
$
if $1 \le p \le q \le \infty$.  Then it is natural to define the
class $A_{\infty,\cb}$ by setting

$$
A_{\infty,\cb}:= \bigcup_{p>1} A_{p,\cb}.
$$

For a general basis $\cb$ we obtain the following strong  type result.
\begin{theorem}\label{t1}
Let $1<p_1,\cdots, p_m<\infty$ and $0<p<\infty$ such that $\f1p=\sum_{j=1}^m \f1{p_j}$.
Assume that $\cb$ is a basis and that $\{\Psi_j\}_{j=1}^m$
is a sequence of Young functions
such that
$$M_{ \overrightarrow{\bar\Psi}}^\cb(\vec f)(x)
:=\sup_{B\in\cb, B\ni x}\prod_{j=1}^m\|f_j\|_{\bar\Psi_j, B}$$
is bounded from $L^{p_1}(\rn)\times\cdots\times L^{p_m}(\rn)$ to $L^p(\rn)$.
Let $(\nu, \vec w\,)=(\nu, w_1,\cdots, w_m)$ be weights such
that $\nu^p$ satisfies condition $(A)$, and that
\begin{equation}\label{m+1-weight}
\sup_{B\in\cb}\lf(\f1{|B|}\int_B \nu(x)^p\,dx\r)^{1/p}
\prod_{j=1}^m \|w_j^{-1}\|_{\Psi_j, B}<\nf.
\end{equation}
Then $\cm_\cb $ is bounded from $L^{p_1}(w_1^{p_1})
\times\cdots\times L^{p_m}(w_m^{p_m})$ to $L^p(\nu^p)$.
\end{theorem}

\begin{remark}\rm\lab{r_t1}
We observe that for all $x\in\rn$ and for all non-negative functions
$\vec f=(f_1,\dots, f_m)$,
$$\cm_{\ovec{\bar\Psi}}^\cb(\vec f)(x)\le\prod_{j=1}^m M_{   \bar \Psi_j }^\cb(f_j)(x).$$
Thus, if we assume that each $M_{\bar\Psi_j}^\cb$ is bounded on $L^{p_j}(\rn)$, then
$\cm_{\ovec{\bar\Psi}}^\cb$ is bounded from
$L^{p_1}(\rn)\times\cdots\times L^{p_m}(\rn)$ to $L^p(\rn)$,
and consequently, the conclusion of Theorem \ref{t1} gives us that
$\cm_\cb$ is bounded from
$L^{p_1}(w_1^{p_1})\times\cdots\times L^{p_m}(w_m^{p_m})$ to $L^p(\nu^p)$
when $(\nu, \vec w)$ satisfies \eqref{m+1-weight}.
\end{remark}

Between these general bases, we are particularly
interested in Muckenhoupt basis introduced in \cite{P1}.
We say that $\cb$ is a {\it Muckenhoupt basis} if for any $1<p<\infty$
and for any $w \in A_{p,\cb}$, $M_\cb$ is bounded in $L^p(w)$.
Most of the important bases are in this class and, in particular,
those mentioned above: $\cq, \cd$, $\crec$.  The fact that $\crec$
is a Muckenhoupt basis can be found in \cite{GCRdF}. The basis $\Re$
is also a Muckenhoupt basis as shown by R. Fefferman \cite{RF2}.

For Muckenhoupt bases, the generalization of the power bump
condition (\ref{bumpedap_R}) assures the boundedness
of $M_{ \overrightarrow{\bar\Psi}}^\cb$. Therefore we can deduce the following result.

\begin{corollary} \label{c1}
Let $\cb$ be a Muckenhoupt  basis.
Let $ \frac1m<p<\infty$ and $1<p_1,\dots,p_m<\infty$ such
that $\f1 p=\f 1 {p_1} + \cdots +\f 1{p_m}$.
If the weights $(\nu, \vec w)= (\nu, w_1,\cdots, w_m)$
satisfy the power bump condition
\begin{equation}\label{bumpedap}
\sup_{B\in \cb} \lf(\f 1 {|B|} \int _B \nu(x)\, dx \r)\, \prod_{j=1}^m
\left( \f 1 {|B|} \int _Bw_j^{(1-p'_j)r}\,
dx\right)^{\frac{p}{p_j'r}}<\infty\,
\end{equation}
for some $r>1$ and
$\nu$ satisfies condition $(A)$,
then
$\cm_\cb$ is bounded from $L^{p_1}(w_1)
\times \cdots \times L^{p_m}(w_m)$ to $L^{p}(\nu).$
\end{corollary}

\begin{proof}
For each $j\in\{1,\cdots, m\}$, we set $\widetilde w_j:= w_j^{1/p_j}$ and
$\Psi_j(t):= t^{p_j'r}$ for all $t\in(0, \infty)$.
Set $\widetilde \nu:=\nu^{1/p}$. Then the power bump condition
\eqref{bumpedap} can be rewritten as
$$\sup_{B\in\cb}\lf\{\frac{1}{|B|}\int_B \widetilde\nu^p\,dx\r\}^{1/p}
\prod_{j=1}^m\|\widetilde w_j^{-1}\|_{\Psi_j, B}<\infty.$$
In this case, for all $x\in\rn$,
$$M_{\bar\Psi_j}^{\cb}f(x)=\sup_{B\in\cb, B\ni x}\|f\|_{\bar\Psi_j, B}
=\sup_{B\in\cb, B\ni x}\lf\{\frac1{|B|}\int_B |f(y)|^{(p_j'r)'}\,dy\r\}^{1/(p_j'r)'}.$$
Since $\cb$ is a Muckenhoupt  basis  and $(p_j'r)'<p_j$,
every $M_{\bar\Psi_j}^{\cb}$ is bounded on
$L^{p_j}(\rn)$. By Remark \ref{r_t1} this implies that
$\cm_{\ovec{\bar\Psi}}^{\cb}$ is bounded from
$L^{p_1}(\rn)\times \cdots \times L^{p_m}(\rn)$ to $L^{p}(\rn)$.
Thus, by Theorem \ref{t1}
$$\cm_\cb: \,  L^{p_1}(\widetilde w_1^{p_1})
\times \cdots \times L^{p_m}(\widetilde w_m^{p_m}) \to L^{p}(\widetilde\nu^p),$$
which completes the proof.
\end{proof}

A result stronger than Corollary \ref{c1} is the following boundedness of the
multilinear strong maximal function, where $(\nu, \vec w)$
satisfy some logarithmic type condition.

\begin{corollary}\lab{c2}
Let $1<p_1,\cdots, p_m<\infty$ and $ \frac1m <p<\infty$ such that $\f1p=\sum_{j=1}^m \f1{p_j}$.
Let $(\nu, \vec w\,)= (\nu, w_1,\cdots, w_m)$
such that $\nu$ and all the $w_j$'s are weights,
and $\nu^p$ satisfies condition (A).
If there is a positive constant $K$ such that for all rectangles $R$,
\[\lf(\f1{|R|}\int_R \nu(x)^p\,dx\r)^{1/p} \prod_{j=1}^m \|w_j^{-1}\|_{\Psi_j, R}<\nf,\]
where every $\Psi_j$ is a Young function
such that $\bar\Psi_j\in B_{p_j}^\ast$. Then $\cm_\crec $
is bounded from $L^{p_1}(w_1^{p_1})\times\cdots\times L^{p_m}(w_m^{p_m})$ to $L^p(\nu^p)$.
\end{corollary}
\begin{proof}
From  Theorem \ref{t2} and the assumption that each $\bar\Psi_j$ is a Young function
satisfying the condition \eqref{Strong-Bp}, it follows that
every $M_{ \bar\Psi_j}^\crec$ is bounded on $L^{p_j}(\rn)$.
Then, applying Remark \ref{r_t1} and Theorem \ref{t1} with $\cb=\crec$,
we obtain the desired conclusion.
\end{proof}

We end this section with the proof of Theorem \ref{Two-weight}
that is a straight consequence of  Corollary \ref{c2}.

\begin{proof}[Proof of Theorem \ref{Two-weight}]
We notice first that (i) is the linear case ($m=1$) of Corollary \ref{c2}.
For the proof of (ii) we proceed as in \cite[Proposition 3.2]{P3}.
That is, consider any non-negative function $g$ and define the couple
of weights $(u,v)=(M_\Phi^\crec(g^{1/p})^{-1},g^{-1/p})$.
 Obviously, $(u, v)$ satisfies condition \eqref{2-weight} with
 constant $K=1$. Hence, by hypothesis there is a constant $C$ such that
\[\int_{\rn} [M_\crec (f)(y)]^p\frac{1}{[M_{\bar \Phi }^\crec(g^{1/p})(y)]^p}\,dy
\le C\int_{\rn} f(y)^p \frac1{g(y)}\,dy.\]
Finally, by Theorem \ref{t2}, this inequality implies that $\Phi\in B_p^*$,
which completes the proof.
\end{proof}


\section{Proof of the strong type estimate in the $(m+1)$-weight case}
To prove Theorem \ref{t1}, we use an argument that
combines ideas from \cite[Theorem~2.5]{GLPT},
the second proof of Theorem 3.7 in \cite{LOPTT}, and some other
tools from \cite{J} and \cite{P3}.
First, we will recall an additional definition for general
bases and a special case of
a lemma from \cite{J} that we will need for the proof of Theorem \ref{t1}.
The following definition concerns the concept of {\it
$\alpha$-scattered families}, which was considered in the works
\cite{JT} and \cite{J} and implicitly in \cite{C} and \cite{cf}.

\begin{definition}\label{d3.1}
Let $\cb$ be a basis and $0<\alpha<1$.  A finite sequence
$\{\tilde{A_i}\}_{i=1}^{M} \subset\cb$ of sets of finite
Lebesgue measure is called $\alpha$--scattered with respect to the
Lebesgue  measure if for all $1<i\leq M$,
$$
 \bigg|\tilde{A_i}\bigcap \bigcup_{s<i}\tilde{A_s}\bigg|\leq \alpha |\tilde{A_i}|.
$$
\end{definition}

The proof of the following lemma
is in \cite[p.\,370, Lemma~1.5]{J}; see also \cite{GLPT}.
\begin{lemma}\label{scatteredproperty}
Let $\cb$ be a basis and let $w$ be a weight associated to this
basis. Suppose further that $w$ satisfies condition $(A)$   for some
$0<\la<1$  and $0<c(\la)<\infty$. Then  given any finite
sequence $\{A_i\}_{i=1}^{M}$ of sets  $A_i\in \cb$,
\begin{enumerate}
\item[(a)] there exists a subsequence $\{\tilde{A_i}\}_{i\in I}$ of
$\{A_i\}_{i=1}^{M}$ which is  $\la$-scattered with respect to the Lebesgue measure;
\item[(b)] $\tilde{A_i}= A_i, \,  i\in I$;
\item[(c)] for any $1\leq i<j\leq M+1$,
\begin{equation*}
w\Big(\bigcup_{s<j}A_s\Big)\leq   c(\la)
\,\Big[w\Big(\bigcup_{s<i}A_s\Big) + w\Big(\bigcup_{i\leq
s<j}\tilde{A_{s}}\Big)   \Big],
\end{equation*}
where $\tilde{A_{s}}=\emptyset$ when $s\notin I$.
\end{enumerate}
\end{lemma}

\begin{proof}[Proof of Theorem \ref{t1}]
Let $N>0$ be a   large  integer.  We  will prove the required
estimate for the quantity
\begin{equation*}
\int_{2^{-N}<\cm_{\cb}(\vec f\,)\leq 2^{N+1}} \cm_{\cb}(\vec f\,)(x)^{p}\, \nu(x)^p\, dx
\end{equation*}
with a bound independent of  $N$.
 We claim that for each integer $k$ with $|k|\leq N$,
there exist a compact set $K_k$ and a finite sequence $b_k=\{B_\al^k\}_{\al\geq 1}$
of sets $B_\al^k \in \cb$ such that
$$
\nu^p(K_k) \leq \nu^p(\{\cm_{\cb}(\vec f\,)>2^k \})\leq 2 \, \nu^p(K_k)
$$
The sequence of sets $\{\cup_{B\in b_k} B\}_{k=-N}^N$ is decreasing. Moreover,
$$\bigcup_{B\in b_k} B\subset K_k \subset  \{\cm_{\cb}(\vec f\,)>2^k \},$$
and
\begin{equation}\label{bigger2^k}
 \prod_{j=1}^m   \f 1 {|B_\al^k|} \int _{B_\al^k} |f_j (y)| \, dy >2^k,
 \qqq \al\geq 1,
\end{equation}
To see the claim, for each $k$ we choose a compact set
$ \widetilde K_k\subset \{\cm_{\cb}(\vec f\,)>2^k \} $
such that
$$\nu^p(\widetilde K_k)\leq \nu^p(\{\cm_{\cb}(\vec f\,)>2^k \})\leq 2 \,
\nu^p(\widetilde K_k).$$
For this $\widetilde K_k$, there exists
a finite sequence $b_k=\{B_\al^k\}_{\al\geq 1}$
of sets $B_\al^k \in \cb$ such that every $B_\al^k$ satisfies
\eqref{bigger2^k} and such that
$\widetilde K_{k} \subset \cup_{B\in b_k} B\subset \{\cm_{\cb}(\vec f\,)>2^k \}.$
Now, we take a compact set $K_k$ such that $\cup_{B\in b_k} B\subset K_k
\subset \{\cm_{\cb}(\vec f\,)>2^k \}$.
Finally, to ensure that  $\{\cup_{B\in b_k} B\}_{k=-N}^N$ is decreasing,
we begin the above selection
from $k=N$ and once a selection is done for $k$ we do the selection
for $k-1$ with the next additional requirement
$$\widetilde K_{k-1} \supset K_k.$$
This proves the claim. Since $\{\cup_{B\in b_k} B\}_{k=-N}^N$
is a sequence of decreasing sets, we set
\begin{equation*}
\Omega_k=
\begin{cases}
\bigcup_{\al}B_\al^k =\bigcup_{B\in b_k} B  &\textup{when }\q |k|\leq N,
\\
\emptyset  &\textup{when }\q |k|>N.
\end{cases}
\end{equation*}
Observe that these sets are decreasing in $k$, i.e.,
$\Omega_{k+1}\subset  \Omega_k$ when $-N<k\le N$.

We now distribute the sets in $\bigcup_k b_k$ over $\mu$ sequences
$\{A_i(l)\}_{i\geq 1}$, $0\leq l\leq \mu-1 $, where $\mu$ will be
chosen momentarily to be an appropriately large natural number. Set
$i_0(0)=1$. In the first $i_1(0)-i_0(0)$ entries
of $\{A_i(0)\}_{i\geq 1}$, i.e., for
$$
 i_{0}(0)\leq i<i_{1}(0),
$$
we place the    elements of the sequence $b_N=\{B_\al^N\}_{\al\geq 1}$
  in the order indicated by the index $\al$. For the next
  $i_{2}(0)-i_{1}(0)$ entries of $\{A_i(0)\}_{i\geq 1}$, i.e., for
$$
 i_{1}(0)\leq i<i_{2}(0),
$$
we place   the elements of the sequence $b_{N-\mu}$. We continue in this way until
we  reach the first integer $m_0$ such that $N-m_0\mu\geq -N$, when
we stop. For indices $i$  satisfying
$$
i_{m_0}(0)\leq i<i_{m_0+1}(0),
$$
we place  in the sequence   $\{A_i(0)\}_{i\geq 1}$ the    elements of
$b_{N-m_0\mu}$. The sequences
$\{A_i(l)\}_{i\geq 1}$, $1\leq l\leq \mu-1,$ are defined similarly,
starting from $b_{N-l}$ and using the families
$b_{N-l-s\mu}$, $s=0,1,\cdots, m_l$,  where $m_l$ is chosen to be the
biggest integer such that $N-l-m_l\mu\geq -N$.

 Since $\nu^p$ is a weight associated to $\cb$ and it satisfies condition (A),
 we can apply Lemma \ref{scatteredproperty} to each $\{A_i(l)\}_{i\geq 1}$ for some
fixed $0<\la<1$. Then we obtain sequences
$$
\{\tilde{A}_i(l)\}_{i\geq 1} \subset \{A_i(l)\}_{i\geq 1} \, , \qqq 0\leq
l\leq \mu-1,
$$
which are $\la$-scattered with respect to the
Lebesgue measure. In view of the definition of the set $\Omega_k$ and
the construction of the families $\{A_i(l)\}_{i\geq 1}$, we can use
assertion (c) of Lemma \ref{scatteredproperty} to obtain that
for any $k=N-l-s\mu$ with $0\le l\le\mu-1$ and $1\le s\le m_l$,
\begin{eqnarray*}
\nu^p (\Omega_k)=\mu^p(\Omega_{N-l-s\mu}) &&\leq  c\Bigg[ \nu^p (\Omega_{k+\mu}) +
\nu^p \lf( \bigcup_{ i_s(l)\leq i<i_{s+1}(l)
}\tilde{A}_i(l)\r) \Bigg]\\
&& \leq c\, \nu^p (\Omega_{k+\mu}) +
c\! \sum_{i=i_{s}(l)}^{i_{s+1}(l)-1} \nu^p (\tilde{A}_i(l)).
\end{eqnarray*}
For the case $s=0$, we have $k=N-l$ and
\begin{eqnarray*}
\nu^p (\Omega_k)=\nu^p(\Omega_{N-l}) && \leq
c\! \sum_{i=i_0(l)}^{i_1(l)-1} \nu^p (\tilde{A}_i(l)).
\end{eqnarray*}
Now, all these sets $\{\tilde{A}_i(l)\}_{i=i_s(l)}^{i_{s+1}(l)-1}$
belong to $b_k$ with  $k= N-l-s\mu $  and therefore
\begin{eqnarray}\label{ee4}
\prod_{\al=1}^m   \f 1 {|\tilde{A}_i(l)|} \int _{\tilde{A}_i(l)}
|f_j (x)| \, dx >2^k.
\end{eqnarray}
 It now readily follows that
$$
\int_{2^{-N}<\cm_{\cb}(\vec f\,)\leq 2^{N+1}} \cm_{\cb}(\vec f\,)(x)^{p}\, \nu^p (x)\, dx
\leq 2^{p} \sum_{k=-N}^{N-1}2^{kp} \nu^p (\Omega_k)
$$
and then
\begin{eqnarray}\label{ee3}
\sum_{k=-N}^{N-1}2^{kp} \nu^p(\Omega_k)
&&=\sum_{\ell=0}^{\mu-1} \sum_{0\le s\le m_\ell} 2^{p(N-l-s\mu)} \nu^p(\Omega_{N-l-s\mu})\\
&&= c \sum_{\ell=0}^{\mu-1}\sum_{1\le s\le m_\ell} 2^{p(N-l-s\mu)} \nu^p(\Omega_{N-l-s\mu+\mu})\nonumber\\
&&\quad+c \sum_{\ell=0}^{\mu-1}\sum_{0\le s\le m_\ell} 2^{p(N-l-s\mu)}
\sum_{i=i_{s}(l)}^{i_{s+1}(l)-1} \nu^p (\tilde{A}_i(l)).\nonumber
\end{eqnarray}
Observe that  the first term in the last equality of \eqref{ee3} is equal to
\begin{eqnarray*}
c2^{-p\mu}\sum_{\ell=0}^{\mu-1}\sum_{0\le s\le m_\ell-1} 2^{p(N-l-s\mu)} \nu^p(\Omega_{N-l-s\mu})
\le c 2^{-p\mu} \sum_{k=-N}^{N-1}2^{kp} \nu^p(\Omega_k).
\end{eqnarray*}
If we choose $\mu$ so large that $c\, 2^{ -\mu p}\leq\f12$ and since
everything involved is finite the first term on the right hand side of \eqref{ee3}
can be subtracted from the left hand side of \eqref{ee3}. This yields
$$
\int_{2^{-N}<\cm_{\crec}(\vec f\,)\leq 2^{N+1}}
\!\!\!\! \! \cm_{\crec}(\vec
f\,)^{p}\, \nu  \, dx \leq  2^{p+1}c\, \sum_{\ell=0}^{\mu-1}\sum_{0\le s\le m_\ell}
\sum_{i=i_{s}(l)}^{i_{s+1}(l)-1} 2^{p(N-l-s\mu)} \nu^p (\tilde{A}_i(l)).
$$
By \eqref{ee4} and the generalized H\"older's inequality \eqref{Holder-Orlicz} we obtain
\begin{eqnarray}\label{ee5}
&&\sum_{\ell=0}^{\mu-1}\sum_{0\le s\le m_\ell}
\sum_{i=i_{s}(l)}^{i_{s+1}(l)-1} 2^{p(N-l-s\mu)} \nu^p (\tilde{A}_i(l))\\
&&\hs\leq  c \sum_{\ell=0}^{\mu-1}\sum_{0\le s\le m_\ell}
\sum_{i=i_{s}(l)}^{i_{s+1}(l)-1} \nu^p (\tilde{A}_i(l))
\left[\prod_{j=1}^m   \f 1 {|\tilde{A}_i(l)|} \int
_{\tilde{A}_i(l)} |f_j |   dx\right]^p  \nonumber\\
&&\hs\leq  c \sum_{\ell=0}^{\mu-1}\sum_{0\le s\le m_\ell}
\sum_{i=i_{s}(l)}^{i_{s+1}(l)-1} \nu^p (\tilde{A}_i(l))
\left[\prod_{j=1}^m \|f_j w_j\|_{\bar\Psi_j, \tilde{A}_i(l)} \|w_j^{-1}\|_{\Psi_j, \tilde{A}_i(l)}\right]^p\nonumber\\
&&\hs  \le c\sum_{\ell=0}^{\mu-1}\sum_{0\le s\le m_\ell}
\sum_{i=i_{s}(l)}^{i_{s+1}(l)-1}
\left[\prod_{j=1}^m \|f_j w_j\|_{\bar\Psi_j, \tilde{A}_i(l)}\right]^p|\tilde{A}_i(l))|,\nonumber
\end{eqnarray}
where in the last step we use the assumption \eqref{m+1-weight}.

For each $l$ we let $I(l)$ be the index set of $\{\tilde{A}_i(l)\}_{0\le s\le m_\ell,\,
 i_{s}(l)\le i<i_{s+1}(l)}$, and
$$E_1(l)= \tilde{A}_1(l) \q \& \q E_i(l)= \tilde{A}_i(l) \setminus \bigcup_{s<i}
\tilde{A}_s(l) \qqq \forall\,i\in I(l).
$$
Since the sequences $\{\tilde{A}_i(l)\}_{i\in
I(l)}$ are $\la$--scattered with respect to the Lebesgue measure, for each $i$
$
|\tilde{A}_i(l)| \leq \frac{1}{1-\la }|E_i(l)|.
$
Then we have the following estimate for \eqref{ee5}
\begin{eqnarray}\label{cont3}
\f{C}{1-\lambda  }   c\sum_{l=0}^{\mu-1}
\sum_{i\in I(l)}
\left[\prod_{j=1}^m \|f_j w_j\|_{\bar\Psi_j, \tilde{A}_i(l)}\right]^p |E_i(l)|.
\end{eqnarray}
The collection $\{E_i(l)\}_{i\in I(l)}$ is a disjoint family, we
can therefore use the fact that $\cm_{\ovec{\bar\Psi}, \cb}$ is
bounded from $L^{p_1}(\rn)\times\cdots\times L^{p_m}(\rn)$
to $L^p(\rn)$  so as to estimate this last equation \eqref{cont3}. Then
\begin{eqnarray*}
&& \sum_{l=0}^{\mu-1} \sum_{i\in I(l)}  \,
 \int _{E_i(l)}   \lf[\cm_{\ovec{\bar\Psi}, \cb} ((f_1w_1,\cdots f_mw_m))(x)\r]^p\, dx \\
&&\hs\leq  c_{\mu} \int _{ \rn}  \lf[\cm_{\ovec{\bar\Psi}, \cb}
((f_1w_1,\cdots f_mw_m))(x)\r]^p\, dx\\
&&\hs\leq C \prod_{j=1}^m   \|f_jw_j\|_{L^{p_j}(\rn)}^p.
 \end{eqnarray*}
Letting $N\to \nf$ yields the desired assertion of the theorem.
\end{proof}

\begin{remark}\rm\label{r_t2}
We point out that the fact that Theorem \ref{t2}(ii) implies Theorem \ref{t2}(iv)
can be deduced
by proceeding as in  the proof of Theorem \ref{t1}.
Indeed, we  will prove the required
estimate for the quantity
\begin{equation*}
\int_{2^{-N}<M_\Phi^\crec(f\,)\leq 2^{N+1}} M_\Phi^\crec(f\,)(x)^{p}\, w(x)\, dx
\ls \int_{\rn} f(y)^p M_\crec w(y)\,dy,
\end{equation*}
where $N$ is a large integer. We use the same covering argument
of Theorem \ref{t1} replacing \eqref{bigger2^k}  by
\[\f 1 {|R_\al^k|} \int _{R_\al^k} |f(y)| \, dy >2^k.\]
Repeating equations (\ref{ee4}) and (\ref{ee3})  , we will get
\begin{eqnarray*}
&&\sum_{\ell=0}^{\mu-1}\sum_{0\le s\le m_\ell}
\sum_{i=i_{s}(l)}^{i_{s+1}(l)-1} 2^{p(N-l-s\mu)} w (\tilde{A}_i(l))\\
&&\hs\leq  c \sum_{\ell=0}^{\mu-1}\sum_{0\le s\le m_\ell}
\sum_{i=i_{s}(l)}^{i_{s+1}(l)-1} w (\tilde{A}_i(l))
\|f\| _{\Phi, \tilde{A}_i(l)}^p  \nonumber\\
&&\hs\leq  c \sum_{\ell=0}^{\mu-1}\sum_{0\le s\le m_\ell}
\sum_{i=i_{s}(l)}^{i_{s+1}(l)-1}
\lf\|f \lf(\frac{w(\tilde{A}_i(l))}{|\tilde{A}_i(l)|}\r)^{1/p}\r\| _{\Phi, \tilde{A}_i(l)}^p
|\tilde{A}_i(l)|\nonumber\\
&&\hs  \le c\sum_{\ell=0}^{\mu-1}\sum_{0\le s\le m_\ell}
\sum_{i=i_{s}(l)}^{i_{s+1}(l)-1}
\lf\|f (M_\crec w)^{1/p}\r\| _{\Phi, \tilde{A}_i(l)}^p
|\tilde{A}_i(l)|,\nonumber
\end{eqnarray*}
where in the penultimate step we used the generalized H\"older's inequality
\eqref{Holder-Orlicz}. Finally, we will obtain the claimed conclusion
using the fact that the operator $M_\Phi^\crec $ is bounded on $L^p(\rn)$. The
details are left to the reader.
\end{remark}



\bibliographystyle{amsalpha}

\end{document}